\documentclass[11pt]{amsart}
\usepackage{amssymb}

\usepackage{xcolor}

\theoremstyle{plain}
\newtheorem{proposition}{Proposition}[section]
\newtheorem{theorem}[proposition]{Theorem}
\newtheorem{lemma}[proposition]{Lemma}
\newtheorem{corollary}[proposition]{Corollary}
\theoremstyle{definition}

\newtheorem{definition}[proposition]{Definition}

\theoremstyle{remark}
\newtheorem{remark}[proposition]{Remark}

\newtheorem*{conject}{Conjecture}
\newtheorem*{question}{Question}
\newtheorem*{rem}{Remark}
\newtheorem*{acknow}{Acknowledgements}

 \numberwithin{equation}{section}
 
\DeclareMathOperator{\Aff}{Aff}

\DeclareMathOperator{\Span}{Span}

\DeclareMathOperator{\Cc}{\mathcal{C}}

\DeclareMathOperator{\Hc}{\mathcal{H}}

\DeclareMathOperator{\Bb}{\mathbb{B}}
\DeclareMathOperator{\Cb}{\mathbb{C}}
\DeclareMathOperator{\Db}{\mathbb{D}}

\DeclareMathOperator{\Kb}{\mathbb{K}}
\DeclareMathOperator{\Nb}{\mathbb{N}}

\DeclareMathOperator{\Xb}{\mathbb{X}}

\newcommand{\abs}[1]{\left|#1\right|}

\newcommand{\norm}[1]{\left\|#1\right\|}
\newcommand{\wt}[1]{\widetilde{#1}}
\newcommand{\wh}[1]{\widehat{#1}}
\newcommand{\ip}[1]{\left\langle #1\right\rangle}

\begin{document}

\title[Domains with negatively pinched K\"ahler metrics]{The geometry of domains with negatively pinched K\"ahler metrics}

\author[F. Bracci]{Filippo Bracci$^1$
}
\address{F. Bracci: Dipartimento di Matematica, Universit\`a di Roma ``Tor Vergata", Via della Ricerca Scientifica 1, 00133, Roma, Italia.} \email{fbracci@mat.uniroma2.it}

\author[H. Gaussier]{Herv\'e Gaussier$^2$
}
\address{H. Gaussier: Univ. Grenoble Alpes, CNRS, IF, F-38000 Grenoble, France}
\email{herve.gaussier@univ-grenoble-alpes.fr}

\author[A. Zimmer]{Andrew Zimmer$^3$
}
\address{A. Zimmer: Department of Mathematics, Louisiana State University, Baton Rouge, LA USA }
\curraddr{Department of Mathematics, University of Wisconsin-Madison, Madison, WI, USA}
\email{amzimmer2@wisc.edu}

\date{\today}

\thanks{$^1\,$Partially supported by PRIN {\it Real and Complex Manifolds: Topology, Geometry and holomorphic dynamics} n.2017JZ2SW5, by INdAM and by the MIUR Excellence Department Project awarded to the  
Department of Mathematics, University of Rome Tor Vergata, CUP E83C18000100006}
\thanks{$^2\,$Partially supported by ERC ALKAGE}
\thanks{$^3\,$Partially supported by the National Science Foundation under grants DMS-1760233, DMS-2105580, and DMS-2104381}

\begin{abstract} We study how the existence of a negatively pinched K{\"a}hler metric on a domain in complex Euclidean space restricts the geometry of its boundary. In particular, we show that if a convex domain admits a complete K{\"a}hler metric, with pinched negative holomorphic bisectional curvature outside a compact set, then the boundary of the domain does not contain any complex subvariety of positive dimension. Moreover, if the boundary of the domain is smooth, then it is of finite type in the sense of D'Angelo. We also use curvature to provide a characterization of strong pseudoconvexity amongst convex domains. In particular, we show that a convex domain with $C^{2,\alpha}$ boundary is strongly pseudoconvex if and only if it admits a complete K{\"a}hler metric with sufficiently tight pinched negative holomorphic sectional curvature outside a compact set.
 \end{abstract} 

\maketitle

\section{Introduction}

Let $(M,J)$ be a complex manifold with K{\"a}hler metric $g$ and let $R(g)$ denote the curvature tensor of $(M,g)$. Then the \emph{holomorphic bisectional curvature} of non-zero $X,Y \in T_pM$ is given by 
\begin{align*}
B(g)(X,Y) = \frac{R(g)(X,JX,Y, JY)}{g(X,X)g(Y,Y)}
\end{align*}
and the \emph{holomorphic sectional curvature} of a non-zero $X \in T_pM$ is given by 
\begin{align*}
H(g)(X) = B(g)(X,X).
\end{align*}

We say that $(M,g)$ has \emph{pinched negative holomorphic bisectional curvature} if there exist $a,b>0$ such that 
\begin{align*}
-a \leq B(g)(X,Y) \leq -b
\end{align*}
for all $p \in M$ and non-zero vectors $X,Y \in T_pM$. 
Likewise, we say that $(M,g)$ has \emph{pinched negative holomorphic sectional curvature} if there exist $a,b>0$ such that 
\begin{align*}
-a \leq H(g)(X) \leq -b
\end{align*}
for all $p \in M$ and non-zero vector $X \in T_pM$. It follows from~\cite[Equation 4]{GK1967}, that if the Riemannian sectional curvature is negatively pinched, then the holomorphic bisectional curvature is also negatively pinched. However, there exist examples of K{\"a}hler manifolds which have negatively pinched holomorphic bisectional curvature, but not negatively pinched Riemannian sectional curvature. 

The holomorphic sectional curvature determines the entire curvature tensor, but in general it is unclear how conditions on the holomorphic (bi)sectional curvature restrict the global complex geometry of a manifold. One important result along these lines is due to P. Yang, who in 1976 proved the following theorem.

\begin{theorem}[P. Yang \cite{Yan76}]\label{yang-thm}Let $\Db \subset \Cb$ denote the unit disk. If $d \geq 2$, then $\Db^d:=\Db \times \cdots \times \Db \subset \Cb^d$ does not admit a complete K{\"a}hler metric with pinched negative holomorphic bisectional curvature.
\end{theorem}

We note that the symmetric metric on the bidisk has pinched negative holomorphic sectional curvature. Theorem~\ref{yang-thm} has been generalized by a number of authors, see for instance~\cite{Mok87, Seo12, SZ08, Zhe93}. In all these works, the pinching condition on the holomorphic bisectional curvature is ``global'', in the sense that it is required to hold at each point. However, it seems more natural to ask for those pinching conditions to hold ``asymptotically'', meaning outside a compact set. 

Our first main result provides a vast generalization of P. Yang's theorem and also connects the existence of a complete K\"ahler metric with pinched negative holomorphic bisectional curvature with classical finite type conditions in several complex variables. 

Let $\Gamma$ be a smooth real hypersurface in $\Cb^d$ and let $r$ be a local defining function for $\Gamma$. For $p \in \Cb^d$, let $C^*(0,p)$ denote the set of germs of non constant holomorphic maps $z$ from $\Cb$ to $\Cb^d$, such that $z(0) = p$. If $g$ is a smooth function defined in a neighborhood of $0 \in \Cb$, we denote by $\nu(g)$ the order of vanishing of the function $g - g(0)$ at the origin. Following~\cite{Dan79},
the {\it type} $\tau(\Gamma,p)$ of $M$ at $p \in \Gamma$ is defined by
$$
\tau(\Gamma,p):=\sup_{z \in C^*(0,p)}\frac{\nu(r \circ z)}{\nu(z)}.
$$
Then the hypersurface $\Gamma$ is of {\it finite type} (in the sense of D'Angelo) if $\tau(\Gamma,p) < \infty$ for every $p \in \Gamma$. 

Our first main result is the following. 

\begin{theorem}\label{type-thm} Suppose that $\Omega \subset \Cb^d$ is a convex domain and $\Omega$ has a complete K{\"a}hler metric with pinched negative holomorphic bisectional curvature outside a (possibly empty) compact subset of $\Omega$. Then:
\begin{itemize} 
\item[(1)] $\Omega$ does not contain any complex affine line, 
\item[(2)] $\partial \Omega$ does not contain any complex subvariety of positive dimension, and
\item[(3)] if $\partial \Omega$ is a $C^\infty$ smooth hypersurface, then $\partial \Omega$ is of finite type in the sense of D'Angelo.
\end{itemize}
\end{theorem}

\begin{rem} \ \begin{enumerate}
\item Notice that we do not assume that $\Omega$ is bounded. Further, parts (1) and (2) do not require that $\partial \Omega$ has any regularity. 
\item The condition that $\Omega$ does not contain any complex affine line implies that the Kobayashi distance on $\Omega$ is non-degenerate and that $\Omega$ is biholomorphic to a bounded domain. 
\end{enumerate}
\end{rem}

Under the hypothesis of Theorem~\ref{type-thm}, the Yau-Schwarz Lemma implies that the K{\"a}hler metric is bi-Lipschitz to the Kobayashi metric, see Lemma~\ref{kob-lem1} below. Further, for bounded convex domains of finite type  (in the sense of D'Angelo), the Kobayashi metric induces a Gromov hyperbolic metric space according to~\cite{Zim16}. So we have the following Corollary of Theorem~\ref{type-thm}.

\begin{corollary} Suppose that $\Omega \subset \Cb^d$ is a bounded convex domain with $C^\infty$ boundary. If $g$ is a complete K{\"a}hler metric on $\Omega$ with pinched negative holomorphic bisectional curvature  outside a compact subset of $\Omega$ and $d$ is the distance induced by $g$, then the metric space $(\Omega, d)$ is Gromov hyperbolic. 
\end{corollary}

Finite type conditions are essential in the study of partial differential equations in complex analysis. For instance, it is a classical result due to D. Catlin~\cite{Cat87} that the boundary of a bounded, smooth, pseudoconvex domain is of finite type if and only if the $\bar{\partial}$-Neumann problem satisfies a subelliptic estimate at each boundary point; this implies in particular the regularity up to the boundary of the canonical solution of the inhomogeneous Cauchy-Riemann equation. Therefore, Theorem~\ref{type-thm} part (3) shows that, for at least smoothly bounded convex domains, the existence of a K{\"a}hler metric with pinched negative holomorphic bisectional curvature has strong analytic implications.

Based on Theorem~\ref{type-thm}, it seems natural to conjecture the following. 

\begin{conject} Suppose that $\Omega$ is a bounded pseudoconvex domain with $C^\infty$ boundary in $\mathbb C^d$, $d \geq 1$. Then there exists a complete K{\"a}hler metric on $\Omega$ with pinched negative holomorphic bisectional curvature in a neighborhood of $\partial \Omega$ if and only if $\partial \Omega$ has finite type. 
\end{conject}

For the existence part of the conjecture, there are a number of results concerning the K\"ahler-Einstein and the Bergman metrics. For bounded strongly pseudoconvex domains, works of S.Y. Cheng and S.T. Yau~\cite{CY80}, P. Klembeck~\cite{Kle78}, and K.T. Kim and J. Yu~\cite{KY96} give precise curvature estimates near the boundary, see Theorem~\ref{thm:curv} below for details. J. Bland~\cite{Bl86} proved that the Riemannian curvature of the K\"ahler-Einstein metric with prescribed negative Ricci curvature on the Th\"ullen domain $\{|z|^2 +|w|^{2p} <1\}$ is negatively pinched for $p \geq 1$. For Reinhardt domains of finite type domains in $\mathbb{C}^2$, S. Fu~\cite{Fu96} proved that the Bergman metric has negatively pinched holomorphic sectional curvatures near the boundary.

We can also use curvature to provide a characterization of strong pseudoconvexity for convex domains. Recall that a domain $\Omega \subset \Cb^d$ with $C^2$ boundary is called \emph{strongly pseudoconvex} if the Levi form of $\partial \Omega$ is positive definite. We then prove the following. 

\begin{theorem}\label{thm:char_str_pconvex}
For any $\alpha \in (0,1)$, $d \geq 2$, and $c > 0$, there exists some $\epsilon = \epsilon(\alpha, d, c) >0$ such that: if $\Omega \subset \Cb^d$ is a bounded convex domain with $C^{2,\alpha}$ boundary, then $\Omega$ is strongly pseudoconvex if and only if there exists a complete K{\"a}hler metric $g$ on $\Omega$ with
\begin{align*}
-c-\epsilon \leq H(g) \leq -c + \epsilon
\end{align*}
outside a (possibly empty) compact subset of $\Omega$.
\end{theorem}

Theorem~\ref{thm:char_str_pconvex} generalizes Theorem~1.11 in~\cite{Zim18} which assumed, in addition, that $g$ and the derivatives of $g$ up to order two are uniformly bounded in terms of the Kobayashi metric. Further, as we will describe below, the conclusion of Theorem~\ref{thm:char_str_pconvex} does not hold for convex domains with $C^2$ boundary. 

Another motivation for Theorem~\ref{thm:char_str_pconvex} comes from the following classical results. Given a bounded strongly pseudoconvex domain $\Omega \subset \Cb^d$,  let $g_{KE,\Omega}$ denote the unique K{\"a}hler-Einstein metric in $\Omega$ with Ricci curvature equal to $-(d+1)$, constructed in~\cite{CY80, MY83}. Also, let $g_{B,\Omega}$ denote the Bergman metric in $\Omega$. Then the holomorphic sectional curvature of these metrics has the following behavior near the boundary. 

\begin{theorem}\cite{CY80,KY96,Kle78}\label{thm:curv} Suppose that $\Omega \subset \Cb^d$ is a bounded strongly pseudoconvex domain with $C^2$ smooth boundary. Then by~\cite{KY96,Kle78}
$$\lim_{z \rightarrow \partial\Omega} H(g_{B,\Omega}) = -\frac{4}{d+1}. $$
Further by~\cite[Corollary 6.6]{CY80}, if $\partial \Omega$ is $C^k$ smooth, with $k \geq \max\{3d + 6,2d + 9\}$, then
$$
\lim_{z \rightarrow \partial\Omega} H(g_{KE, \Omega})=-2.
$$
\end{theorem}

To be precise, the expression
$$
\lim_{z \rightarrow \partial\Omega} H(g) = a
$$
means 
$$
\lim_{z \rightarrow \partial\Omega} \sup_{X \in T_z \Omega\setminus\{0\}} \abs{H(g)(X)-a}=0.
$$

Based on Theorem \ref{thm:char_str_pconvex} and Theorem~\ref{thm:curv}, it seems natural to ask the following question. 

\begin{question} Suppose that $\Omega \subset \Cb^d$ is a bounded domain with $C^k$ boundary, $k>2$. If there exists a complete K{\"a}hler metric $g$ on $\Omega$ and a constant $c > 0$ such that 
\begin{align*}
\lim_{z \rightarrow \partial\Omega} H(g) = -c,
\end{align*}
is $\Omega$ strongly pseudoconvex?
\end{question}

The answer to the above question is no when $k=2$: J.E. Forn\ae ss and E. Wold~\cite{FW16}  constructed a bounded convex domain $\Omega$ with $C^2$ boundary which is not strongly pseudoconvex and whose squeezing function tends to one at the boundary. It follows from Theorem~1.1 in~\cite{Zha17}  and Theorem~4 in~\cite{Gon18} that
\begin{itemize}
\item $\lim_{z \rightarrow \partial \Omega} H(g_{B,\Omega}) = -4/(d+1)$,
\item $\lim_{z \rightarrow \partial \Omega} H(g_{KE, \Omega}) = -2$.
\end{itemize}

The paper is organized as follows. In Section~\ref{sec:outline}, we give an outline of the proof of Theorem~\ref{type-thm}, presenting the major results (Theorem~\ref{thm:affine} and Theorem~\ref{thm:construcing_embedded_products}) needed for the proofs of Theorems~\ref{type-thm} and~\ref{thm:char_str_pconvex}. In Section~\ref{sec:prel}, we fix notations and state some classical results on the Kobayashi metric, that will be used later in the paper. In Section~\ref{sec:topology}, we establish some topological properties of the space of all convex domains. 
 We prove Theorem~\ref{thm:construcing_embedded_products}  in Section~\ref{sec:pf_construcing_embedded_products}. In Section~\ref{section:normal}, we prove a compactness result for complete K\"ahler metrics with bounded geometry (in the sense of S.Y. Cheng and S.T. Yau) which are uniformly bi-Lipschitz to the Kobayashi metric. In Section~\ref{section:ricci}, we describe how classical results about the Ricci flow can be used to deform a complete K\"ahler-Einstein metric with negatively pinched holomorphic (bi)sectional curvature to obtain a new metric with bounded geometry. Section~\ref{sec:pf_of_thm_affine} is devoted to the proof of Theorem~\ref{thm:affine}. Finally, we prove Theorem~\ref{type-thm} in Section~\ref{sec:pf_of_main_thm} and Theorem~\ref{thm:char_str_pconvex} in Section~\ref{sec:char_str_convex}.
 
\begin{acknow}
The authors wish to thank Harish Seshadri for very interesting and fruitful discussions concerning the subject of the paper.
\end{acknow}

\section{Outline of the proof of Theorem~\ref{type-thm}}\label{sec:outline}

Showing that $\Omega$ does not contain any complex affine lines is a straight forward consequence of the Yau-Schwarz lemma. A key idea in the proof of the other two assertions is to consider the space of convex domains and the action of the affine group on this space. 

\begin{definition} Let $\Xb_d$ be the set of non-empty convex domains in $\Cb^d$ which do not contain a complex affine line and let $\Xb_{d,0}$ be the set of pairs $(\Omega,x)$ where $\Omega \in \Xb_d$ and $x \in \Omega$. \end{definition}

The sets $\Xb_d$ and $\Xb_{d,0}$ have a natural topology which we describe in Section~\ref{sec:topology}. Let $\Aff(\Cb^d)$ denote the group of affine automorphisms of $\Cb^d$. Then $\Aff(\Cb^d)$ acts on $\Xb_d$ and $\Xb_{d,0}$ in a natural way: 
\begin{align*}
A \cdot \Omega = A\Omega \text{ and } A \cdot (\Omega, z) = (A\Omega, Az).
\end{align*}
Throughout the paper we will study the complex geometry of a convex domain $\Omega$ by considering the domains in the closure of $\Aff(\Cb^d) \cdot \Omega$ in $\Xb_d$. To that end we introduce the following set:

\begin{definition} Given some $\Omega \in \Xb_d$, a convex domain $\Omega_\infty$ in $\Xb_d$ is an \emph{affine limit of $\Omega$} if there exist a sequence of points $z_n \in \Omega$, a point $z_\infty \in \Omega_\infty$, and affine maps $A_n \in \Aff(\Cb^d)$ such that 
\begin{enumerate}
\item $\{z_n\}$ is compactly divergent in $\Omega$ (that is, for every compact subset $K \subset \Omega$ there exists some $N > 0$ such that $z_n \notin K$ for all $n \geq N$),
\item $A_n(\Omega, z_n)$ converges to $(\Omega_\infty, z_\infty)$. 
\end{enumerate}
Let ${\rm AL}(\Omega)\subset \Xb_d$ denote the set of all affine limits of $\Omega$. \end{definition}

The domains in ${ \rm AL}(\Omega)$ reflect the asymptotic geometry of $\Omega$. In many cases, if $\Omega$ has some property in a neighborhood of $\partial \Omega$ then any domain in ${ \rm AL}(\Omega)$ has that same property globally (see Theorem~\ref{thm:affine} below). Further, one can sometimes construct a domain in ${ \rm AL}(\Omega)$ with very nice properties (see Theorem~\ref{thm:construcing_embedded_products} below). 

 The first main step in our proofs is showing that the existence of a K{\"a}hler metric with pinched negative curvature is preserved under taking limits in $\Aff(\Cb^d)$.
 
\begin{theorem}\label{thm:affine} 
Suppose that $\Omega \in \Xb_d$, $g$ is a complete K{\"a}hler metric on $\Omega$, and $T(g)$ is either $H(g)$ or $B(g)$. Assume there exist $0 < b < a$ such that
\begin{align*}
-a \leq T(g)  \leq -b
\end{align*}
outside a compact subset $K$ of $\Omega$.
If $\Omega_\infty \in { \rm AL}(\Omega)$ and $\epsilon > 0$, then there exists a complete K{\"a}hler metric $g_\infty$ on $\Omega_\infty$ with
\begin{align*}
-a-\epsilon \leq T(g_\infty)  \leq -b+\epsilon
\end{align*}
on $\Omega_\infty$. 
\end{theorem}

A refined version of Theorem~\ref{thm:affine}, that will be needed to prove Theorem~\ref{thm:char_str_pconvex}, will be established in Section~\ref{sec:pf_of_thm_affine} (see Theorem~\ref{thm:affine_closures}).

The second main step is constructing affine limits  with embedded copies of $\Db \times \Db$.

\begin{theorem}\label{thm:construcing_embedded_products} Suppose that $\Omega \in \Xb_d$. If either
\begin{enumerate}
\item there exists a non-constant holomorphic map $\Db \rightarrow \partial\Omega$, or
\item $\partial \Omega$ is a $C^\infty$ hypersurface and $\partial \Omega$ has a point of infinite type,
\end{enumerate}
then there exist $\Omega_\infty \in { \rm AL}(\Omega)$ and a complex affine 2-plane $V$ such that $V \cap \Omega$ is biholomorphic to $\Db \times \Db$.
\end{theorem}

The final step is to use a result of F. Zheng~\cite{Zhe94} to prove the following.

\begin{proposition}\label{prop:no_embedded_products} Suppose that $\Omega \in \Xb_d$  and there exists a complex affine 2-plane $V$ such that $V \cap \Omega$ is biholomorphic to $\Db \times \Db$. Then $\Omega$ does not admit a complete K{\"a}hler metric with pinched negative holomorphic bisectional curvature.
\end{proposition}

\section{Preliminaries}\label{sec:prel}

Let us first fix some notations.
\begin{enumerate}
\item For $z \in \Cb^d$ let $\norm{z}$ be the standard Euclidean norm. 
\item For $z_0 \in \Cb^d$ and $r > 0$, let 
\begin{align*}
\Bb_d(z_0;r) := \left\{ z \in \Cb^d : \norm{z-z_0} < r\right\}.
\end{align*}
Then let $\Bb_d  := \Bb_d(0;1)$ and $\Db := \Bb_1$. 
\end{enumerate}

For a domain $\Omega \subset \Cb^d$, let $k_\Omega$ denote the Kobayashi (pseudo)metric and let $K_\Omega$ denote the Kobayashi (pseudo)distance on $\Omega$. A nice introduction to the Kobayashi metric and its properties can be found in~\cite{Kob88}. 

If $K_{\Omega}$ is a distance, then $\Omega$ is called (Kobayashi) hyperbolic. Every bounded domain is (Kobayashi) hyperbolic. However, without restriction on the geometry of $\Omega$, there is no known characterization of when $K_\Omega$ is a distance (or Cauchy complete).
For convex domains we have the following result of T.J. Barth.

\begin{theorem}[T.J. Barth~\cite{Bar80}]\label{thm:barth}
Suppose that $\Omega \subset \Cb^d$ is a convex domain. Then the following are equivalent:
\begin{enumerate}
\item $\Omega$ does not contain any complex affine lines (i.e. $\Omega \in \Xb_d$),
\item $\Omega$ is (Kobayashi) hyperbolic,
\item $(\Omega, K_\Omega)$ is a proper geodesic metric space. 
\end{enumerate}
\end{theorem}

We will use the following estimate on the infinitesimal Kobayashi metric. For a domain $\Omega \subsetneq \Cb^d$, a point $z \in \Omega$, and  a vector $v \in \Cb^d\backslash \{0\}$ we define
\begin{align*}
\delta_{\Omega}(z) := \inf\{ \norm{z-\zeta} : \zeta \in \partial \Omega\}
\end{align*}
and
\begin{align*}
\delta_{\Omega}(z;v) := \inf\{ \norm{z-\zeta} : \zeta \in (z+\Cb \cdot v) \cap \partial \Omega\}.
\end{align*}

Then the following estimate is well known (see for instance~\cite[Theorem 4.1]{BP95}). 

\begin{lemma}\label{lem:basic_est} If $\Omega \subset \Cb^d$ is a convex domain, $z \in \Omega$, and $v \in \Cb$, then 
\begin{align*}
\frac{\norm{v}}{2\delta_{\Omega}(z;v)} \leq k_{\Omega}(z;v) \leq \frac{\norm{v}}{\delta_{\Omega}(z;v)} \leq \frac{\norm{v}}{\delta_{\Omega}(z)}.
\end{align*}
\end{lemma}

\section{The space of convex domains}\label{sec:topology}  

The \emph{Hausdorff distance} between two bounded non-empty sets $A,B \subset \Cb^d$ is given by
\begin{align*}
d_H(A,B) = \max\left\{\sup_{a \in A}\inf_{b \in B} \norm{a-b}, \sup_{b \in B} \inf_{a \in A} \norm{a-b} \right\}.
\end{align*}

The sets $\Xb_d$ and $\Xb_{d,0}$ can be given a topology from the local Hausdorff semi-norms. For $R >0$ and a set $A \subset \Cb^d$, let $A^{(R)} := A \cap \Bb_d(0;R)$. Then define the \emph{local Hausdorff semi-norms} by
\begin{align*}
d_H^{(R)}(A,B) := d_H(A^{(R)}, B^{(R)}).
\end{align*}
We say a sequence $\Omega_n$ in $\Xb_d$ converges to $\Omega$ in $\Xb_d$ if  there exists some $R_0 \geq 0$ so that 
\begin{align*}
\lim_{n \rightarrow \infty} d_H^{(R)}(\Omega_n,\Omega) = 0
\end{align*}
for all $R \geq R_0$. Further, we say a sequence $(\Omega_n, z_n)$ in $\Xb_{d,0}$ converges to $(\Omega,z)$ in $\Xb_{d,0}$ if $\Omega_n$ converges to $\Omega$ in $\Xb_d$ and $z_n$ converges to $z$.

The action of the affine group $\Aff(\Cb^d)$ is obviously not transitive on $\Xb_{d,0}$, but the following result of S. Frankel shows that the quotient $\Aff(\Cb^d) \backslash \Xb_{d,0}$ is ``compact''.

\begin{theorem}[S. Frankel~\cite{Fra91}] There exists a compact set $K \subset \Xb_{d,0}$ such that $\Aff(\Cb^d) \cdot K = \Xb_{d,0}$. \end{theorem}

As an immediate corollary we obtain the following. 

\begin{corollary} If $\Omega \in \Xb_d$ and $z_n \in \Omega$ is a sequence, then there exist $n_j \rightarrow \infty$ and $A_j \in \Aff(\Cb^d)$ such that $A_j(\Omega, z_{n_j})$ converges in $\Xb_{d,0}$. In particular, the set ${ \rm AL}(\Omega)$ is non-empty. 
\end{corollary}

The next result shows the stability of $k_{\Omega}$  and $K_{\Omega}$ when using this notion of convergence of domains.

\begin{theorem}\label{conv-thm} If $\Omega_n$ converges to $\Omega$ in $\Xb_d$, then 
\begin{align*}
\lim_{n \rightarrow \infty} k_{\Omega_n} = k_\Omega \text{ and } \lim_{n \rightarrow \infty} K_{\Omega_n} = K_\Omega,
\end{align*}
locally uniformly on compact sets. 
\end{theorem}

Theorem~\ref{conv-thm} is well known, see for instance the proof of ~\cite[Theorem 4.1]{Zim16}, but we provide a complete proof in Appendix~\ref{app:pf_of_conv-thm}. 

We will also need the following explicit compact set in $\Xb_d$. Define 
\begin{align*}
\Db_1 := \{ z \in \Cb : \abs{ { \rm Re}(z)} +  \abs{ { \rm Im}(z)} < 1 \}.
\end{align*}
Let $e_1, \dots, e_d$ be the standard basis of $\Cb^d$. Then let 
\begin{align*}
Z_1 := \Span_{\Cb}\{e_2, \dots, e_d\}
\end{align*}
and for  $2 \leq j \leq d$ consider the complex $(d-j)$-dimensional affine plane
\begin{align*}
Z_j :=  ie_1+e_j + \Span_{\Cb}\{e_{j+1}, \dots, e_d\}.
\end{align*}

\begin{proposition}\label{prop:compact_set} Let $\Kb_d$ denote the set of all convex domains $\Omega \in \Xb_d$ where 
\begin{enumerate}
\item $Z_i \cap \Omega =\emptyset$ and
\item $ie_1+ \Db_1 \cdot e_i \subset \Omega$
\end{enumerate}
for all $i=1,\dots, d$. Then $\Kb_d$ is a compact subset of $\Xb_d$. 
\end{proposition}

\begin{proof} This is essentially the proof of Theorem 2.5 in~\cite{Zim17}.
\end{proof}

\section{Proof of Theorem~\ref{thm:construcing_embedded_products}}

A subset $A \subset \Cb^d$ is called a \emph{non-trivial affine disk} if there exists a non-constant affine map $\ell : \Cb \rightarrow \Cb^d$ such that $\ell(\Db) = A$. The key step in the proof of Theorem~\ref{thm:construcing_embedded_products} is establishing the following. 

\begin{theorem}\label{thm:construcing_embedded_products_special_case} Suppose that $\Omega \in \Xb_2$. If $\partial \Omega$ contains a non-trivial affine disk, then there exists some $\Omega_\infty \in { \rm AL}(\Omega)$ which is biholomorphic to $\Db \times \Db$. 
\end{theorem}

\begin{proof} Without loss of generality we can assume that 
\begin{enumerate}
\item $\Omega \subset \{ (z_1,z_2) \in \Cb^2 : { \rm Im}(z_1) > 0\}$,
\item $\{0\} \times \Db \subset \partial \Omega$, and
\item $(i,0) \in \Omega$.
\end{enumerate}

For every $n$, let $z_n := (i/n,0) \in \Omega$. Then pick
\begin{align*}
\xi_n \in (\{i/n\} \times \Cb) \cap  \partial \Omega 
\end{align*}
such that 
\begin{align*}
\norm{\xi_n-z_n} = \inf\left\{ \norm{\xi-z_n} : \xi \in  (\{i/n\} \times \Cb) \cap  \partial \Omega \right\}.
\end{align*}
Since $\overline{\Omega}$ does not contain any complex affine lines, we must have
\begin{align*}
\limsup_{n \rightarrow \infty} \norm{\xi_n-z_n} < +\infty.
\end{align*}
Suppose $\xi_n = ( i/n, a_n)$. By passing to a subsequence we can suppose that $a_n \rightarrow a$. Then 
\begin{align*}
\lim_{n \rightarrow \infty} \norm{\xi_n-z_n}=\lim_{n \rightarrow \infty} \abs{a_n} = \abs{a}
\end{align*}
and $(0,a) \in \partial \Omega$. Since $\{0\} \times \Db \subset \partial \Omega$ and $\Omega$ is convex, we also have $\abs{a} \geq 1$. 

Then consider the matrix
\begin{align*}
A_n := \begin{pmatrix} n & 0 \\ 0 & a_n^{-1} \end{pmatrix}.
\end{align*}
By construction, $A_n(\Omega, z_n) \in \Kb_2$ where $\Kb_2 \subset \Xb_2$ is the subset from Proposition~\ref{prop:compact_set}. So by passing to a subsequence we can assume that $A_n\Omega$ converges to some $\Omega_1$ in $\Xb_2$ (in fact in $\Kb_2$). 

Let $C_2 \subset \Cb$ be the open convex set such that 
\begin{align*}
\{0\} \times \overline{C_2}  = (\{0\} \times \Cb) \cap \partial \Omega.
\end{align*}
Then define $D_2 := a^{-1} \cdot C_2$. Also let $\Hc := \{ z \in \Cb : { \rm Im}(z) >0\}$. \\

\noindent \textbf{Claim:} $\{0\} \times D_2 \subset  \partial \Omega_1$ and $\Omega_1 \subset \Hc \times D_2$. 

\begin{proof}[Proof of Claim:] 
If $(x,y) \in \Omega_1$, then there exists $(x_n, y_n) \in \Omega$ such that $A_n(x_n,y_n) \rightarrow (x,y)$. Thus $nx_n \rightarrow x$ and $y_n/a_n \rightarrow y$. So $x_n \rightarrow 0$ and $y_n \rightarrow ay$. Thus $y \in  a^{-1} \cdot C_2$. So  $\Omega_1 \subset \Cb \times D_2$. Since $\Omega \subset \Hc \times \Cb$ we also have $\Omega_1 \subset \Hc \times \Cb$. So 
\begin{align*}
\Omega_1 \subset (\Cb \times D_2) \cap (\Hc \times \Cb) = \Hc \times D_2.
\end{align*}

Since $(a_n^{-1}\cdot C_2) \times \{ 0 \} \subset A_n \partial \Omega$ and $a_n^{-1} \rightarrow a^{-1}$, the definition of the local Hausdorff topology implies that $\{0\} \times D_2 \subset \overline{\Omega}_1$. Since $\Omega_1 \subset \Hc \times D_2$, we must have $\{0\} \times D_2 \subset  \partial \Omega_1$.
\end{proof}

Let $C_1 \subset \Cb$ be the open convex set such that 
\begin{align*}
C_1 \times \{0\} = (\Cb \times \{ 0\}) \cap \Omega.
\end{align*}
Next define $D_1 := \cup_{n =1}^\infty n \cdot C_1$. Then $D_1$ is a non-empty convex open cone since $0 \in \partial C_1$. \\

\noindent \textbf{Claim:} $D_1 \times D_2 \subset \Omega_1$. 

\begin{proof}[Proof of Claim:]
By construction 
\begin{align*}
 (nC_1) \times \{0\} \subset A_n \Omega
\end{align*}
so, by the definition of the local Hausdorff topology, $D_1 \times \{0\} \subset \overline{\Omega}_1$. Now suppose that $(x,y) \in D_1 \times D_2$. Since $D_1$ is a cone, $(nx,0) \in \overline{\Omega}_1$ for all $n$. Further, the previous claim implies that $(0,y) \in \overline{\Omega}_1$. Thus by convexity 
\begin{align*}
(x,y) = \lim_{n \rightarrow \infty} \frac{1}{n}(nx,0) + \frac{n-1}{n}(0,y) \in \overline{\Omega}_1.
\end{align*}
Thus $D_1 \times D_2 \subset \overline{\Omega}_1$. Since $\Omega_1$ is convex and open in $\Cb^2$, we have $D_1 \times D_2 \subset \Omega_1$. 
\end{proof}

Next consider the matrices
\begin{align*}
B_n := \begin{pmatrix} n & 0 \\ 0 & 1 \end{pmatrix}.
\end{align*}
Then since $D_1$ and $\Hc$ are cones we have
\begin{align*}
D_1 \times D_2 \subset B_n\Omega_1 \subset \Hc \times D_2. 
\end{align*}
So by passing to a subsequence we can assume that $B_n\Omega_1$ converges to some $\Omega_2$ in $\Xb_2$. \\

\noindent \textbf{Claim:} $\Omega_2 = D_1 \times D_2$ and hence, by the Riemann mapping theorem, $\Omega_2$ is biholomorphic to $\Db \times \Db$. 

\begin{proof}[Proof of Claim:]
Notice that $D_1 \times D_2 \subset \Omega_2$ since $D_1 \times D_2 \subset B_n\Omega_1$ for any $n$. 

For every $z \in D_2$ let $S_{z} \subset \Cb$ be the convex open set such that 
\begin{align*}
S_{z} \times \{z\}= (\Cb \times\{z\} ) \cap \Omega_1.
\end{align*}
Then define $\Cc_z := \cup_{n \in \Nb} n \cdot S_{z}$.  Then $\Cc_z$ is a convex open cone since $0 \in \partial S_z$. Further
\begin{align*}
\Cc_z \times \{z\}= (\Cb \times\{z\} ) \cap \Omega_2.
\end{align*}
Since $D_1 \times D_2 \subset \Omega_2$ we see that $\overline{D}_1 \subset \overline{\Cc}_z$. Suppose, for a contradiction, that $\overline{D}_1 \neq \overline{\Cc}_z$ for some $z \in D_2$. Then there exists some $w \in \Cc_z \setminus D_1$. Then $(tw,z) \in \Omega_2$ for all $t > 0$. Then by convexity
\begin{align*}
(w,0) = \lim_{n \rightarrow \infty} \frac{1}{n}(nw, z) + \frac{n-1}{n}(0,0) \in \overline{\Omega}_2.
\end{align*}
So $w \in \overline{D}_1$. So we have a contradiction. Thus $\Cc_z = D_1$ for all $z \in D_2$ and hence $\Omega_2 = D_1 \times D_2$. 
\end{proof}

Finally, since ${ \rm AL}(\Omega_1) \subset { \rm AL}(\Omega)$ we see that $\Omega_2$ is in ${ \rm AL}(\Omega)$.

\end{proof}

Next we recall a number of results which allow us to reduce Theorem~\ref{thm:construcing_embedded_products} to Theorem~\ref{thm:construcing_embedded_products_special_case}. First, a result of S. Frankel allows us to reduce to the case where $d=2$. 

\begin{theorem}\cite[Theorem 9.3]{Fra91}\label{thm:frankel} Suppose that $\Omega \in \Xb_d$ and $V$ is a complex affine $k$-plane intersecting $\Omega$. Let $D = \Omega \cap V$ and suppose there exist affine maps $A_n \in \Aff(V)$ such that $A_n(D)$ converges to $D_\infty$ in $\Xb_k$. Then there exist affine maps $B_n \in \Aff(\Cb^d)$ such that $B_n(\Omega)$ converges to $\Omega_\infty$ in $\Xb_d$ with
\begin{align*}
\Omega_\infty \cap V = D_\infty.
\end{align*}
\end{theorem}

The next two results will allow us to reduce to the case where the boundary contains a non-trivial affine disk. 
 
\begin{proposition}\cite[Lemma 9.5]{Zim17b}\label{prop:rescale_infinite_type} Suppose that $\Omega \in \Xb_2$ has smooth boundary. If $\partial \Omega$ has a point of infinite type,  then there exists some $\Omega_\infty \in { \rm AL}(\Omega)$ such that $\partial \Omega_\infty$ contains a non-trivial affine disk. 
\end{proposition}

\begin{proposition}\cite[Theorem 1.1]{FuSt98}\label{prop:holo_implies_affine} Suppose that $\Omega \in \Xb_d$. If there exists a non-constant holomorphic map $\Db \rightarrow \partial \Omega$, then $\partial \Omega$ contains a non-trivial affine disk. 
\end{proposition}

We can now prove Theorem~\ref{thm:construcing_embedded_products}.

\begin{proof}[Proof of Theorem~\ref{thm:construcing_embedded_products}]\label{sec:pf_construcing_embedded_products} By Theorem~\ref{thm:frankel} we can assume that $d=2$. So suppose that $\Omega \subset \Cb^2$ is a convex domain and either
\begin{enumerate}
\item $\Omega \in \Xb_2$ and there exists a non-constant holomorphic map $\Db \rightarrow \partial\Omega$ or
\item $\partial \Omega$ is $C^\infty$ and $\partial \Omega$ has a point of infinite type.
\end{enumerate}

In the first case, $\partial \Omega$ contains a non-trivial affine disk by Proposition~\ref{prop:holo_implies_affine}. Then Theorem~\ref{thm:construcing_embedded_products_special_case} implies that there exists some $\Omega_\infty \in { \rm AL}(\Omega)$ which is biholomorphic to $\Db \times \Db$. 

In the second case, Proposition~\ref{prop:rescale_infinite_type} implies that there exists some $\Omega_1 \in { \rm AL}(\Omega)$ such that $\partial \Omega_1$ contains a non-trivial affine disk. Then Theorem~\ref{thm:construcing_embedded_products_special_case} implies that there exists some $\Omega_\infty \in { \rm AL}(\Omega_1)$ which is biholomorphic to $\Db \times \Db$. Then $\Omega_\infty \in {\rm AL}(\Omega)$ since ${ \rm AL}(\Omega_1) \subset { \rm AL}(\Omega)$.
\end{proof}

\section{Normal families of K{\"a}hler metrics}~\label{section:normal}

This section is devoted to the proof of the following.

\begin{proposition}\label{prop:rescaling_metrics} Suppose that $\Omega_n$ converges to $\Omega_\infty$ in $\Xb_d$. Further suppose that $g_n$ is a K{\"a}hler metric on $\Omega_n$ such that:
\begin{enumerate}
\item there exists $A > 1$, independent of $n$, such that 
\begin{align*}
\frac{1}{A} k_{\Omega_n}(z;v) \leq \sqrt{g_n(z)(v,v)} \leq A k_{\Omega_n}(z;v) 
\end{align*}
for all $z \in \Omega_n$ and $v \in \Cb^d$,
\item for every $q \geq 0$ there exists $C_q > 0$, independent of $n$, such that 
\begin{align*}
\sup_{\Omega_n} \norm{\nabla^q R(g_n)}_{g_n} \leq C_q.
\end{align*}
\end{enumerate}
Then after passing to a subsequence the metrics $g_n$ converge locally uniformly in the $C^\infty$ topology to a K\"ahler metric $g_\infty$ on $\Omega_\infty$. 
\end{proposition} 

The proof requires the notion of \emph{quasi-bounded geometry}, which was introduced by S.Y. Cheng and S.T. Yau in \cite{CY80}.

\begin{definition}\label{defn:quasi_bd_geom} A K{\"a}hler manifold $(M,g)$ of complex dimension $d$ is said to have \emph{quasi-bounded geometry} if there exist constants $r_2 > r_1 > 0$, $C>1$, and a sequence $\{A_q\}_{q \geq 0}$ of positive numbers such that: for every point $m \in M$ there is a domain $U \subset \Cb^n$ and a nonsingular holomorphic map $\psi:U \rightarrow M$ satisfying the following properties
\begin{enumerate}
\item $\psi(0)=m$, 
\item $\Bb_d(0;r_1) \subset U \subset \Bb_d(0;r_2)$,
\item $C^{-1} g_{\Cb^d} \leq \psi^* g \leq C g_{\Cb^d}$ where $g_{\Cb^d}$ is the standard Euclidean metric on $\Cb^d$, 
\item for every integer $q \geq 0$ 
\begin{align*}
\sup_{x \in U} \abs{ \frac{ \partial^{\abs{\mu}+\abs{\nu}} ((\psi^* g)_{i\overline{j}})}{\partial z^\mu \partial \bar{z}^{\nu}}(x)} \leq A_q \text{ for all } \abs{\mu}+\abs{\nu} \leq q, \ 1 \leq i,j \leq d
\end{align*}
where  $(\psi^* g)_{i\overline{j}}$ is the component of $\psi^*g$ in terms of the canonical coordinates $z=(z_1,\dots, z_d)$ on $\Cb^d$ and $\mu,\nu$ are multiple indices with $\abs{\mu}=\mu_1+\dots+\mu_d$.
\end{enumerate}
The map $\psi$ is called a \emph{quasi-coordinate map} and the pair $(U,\psi)$ is called a \emph{quasi-coordinate chart} of $M$. 
\end{definition}

We will use the following theorem of D. Wu and S.T. Yau.

\begin{theorem}\label{thm:nec_suff_quasi_bd_geom}\cite[Theorem 9]{WY17} Let $(M,g)$ be a complete K{\"a}hler manifold of complex dimension $d$. Then $(M,g)$ has quasi-bounded geometry if and only if for every integer $q \geq 0$ there exists a constant $C_q > 0$ such that the curvature tensor $R(g)$ of $g$ satisfies
\begin{align*}
\sup_{M} \norm{ \nabla^q R(g)}_{g}  \leq C_q.
\end{align*}
Moreover, one can choose the constants $r_1$, $r_2$, $C$, and $\{ A_q\}_{q \geq 0}$ in Definition~\ref{defn:quasi_bd_geom} to depend only on $\{ C_q\}_{q \geq 0}$ and $d$. 
\end{theorem}

\begin{proof}[Proof of Proposition~\ref{prop:rescaling_metrics}]
We first show that after passing to a subsequence the metrics $g_n$ converge locally uniformly in the $C^\infty$ topology to an indefinite Hermitian metric on $\Omega_{\infty}$. To accomplish this  it is enough to show the following: for every compact set $K \subset \Omega_\infty$ and multi-indices $\mu,\nu$ there exist $N \geq 0$ and $C(K,\mu,\nu) > 1$, depending only on $K$, and $\mu$, $\nu$, such that 
\begin{align*}
\sup_{n \geq N} \max_{z \in K} \abs{ \frac{ \partial^{\abs{\mu}+\abs{\nu}} (g_n)_{i\overline{j}}}{\partial z^\mu \partial \overline{z}^{\nu}}(z)} \leq C(K,\mu,\nu).
\end{align*}
Suppose for a contradiction that there exist $K$ and $\mu,\nu$ where no such $N$ and $C(K,\mu,\nu)$ exist. Then for each $k \in \Nb$ there exist $n_k \geq k$ and $z_k \in K$ such that 
\begin{align*}
\abs{ \frac{ \partial^{\abs{\mu}+\abs{\nu}} (g_{n_k})_{i\overline{j}}}{\partial z^\mu \partial \bar{z}^{\nu}}(z_k)} \geq k.
\end{align*}
By passing to a subsequence we can assume that $z_k \rightarrow z_\infty \in \Omega_\infty$. By passing to another subsequence we can assume that $K \subset \Omega_{n_k}$ for all $k \geq 1$. 

By Theorem~\ref{thm:nec_suff_quasi_bd_geom}, each $(M,g_n)$ has quasi-bounded geometry with constants independent of $n$. So for every $k$ there is a domain $U_k \subset \Cb^d$ and a non-singular holomorphic map $\psi_k : U_k \rightarrow \Omega_{n_k}$ with $\psi_k(0)=z_k$ which satisfies the conditions in Definition~\ref{defn:quasi_bd_geom} with uniform parameters $r_1$, $r_2$, $C$, and $\{ A_q\}_{q \geq 0}$.

Fix some $r < r_1$. Then, since $\Bb_d(0;r_1) \subset U_k$ for every $k$, we have 
$$
\begin{array}{lll}
\sup_{w \in  \psi_k(\Bb_d\left(0;r)\right)} K_{\Omega_{n_k}}(z_k,w) & \leq & \sup_{\zeta \in \Bb_d\left(0;r\right)}K_{U_k}(0,\zeta)\\
 & \leq & \sup_{\zeta \in \Bb_d\left(0;r\right)}K_{\Bb_d(0;r_1)}(0,\zeta)\\
 & = & K_{\Bb_d}(0,r/r_1).
\end{array}
$$ 
Theorem~\ref{conv-thm} implies that $K_{\Omega_{n_k}}$ converges locally uniformly to $K_{\Omega_\infty}$ and hence the maps $\psi_{k}$ are uniformly bounded on the ball $\Bb_d(0;r)$. Since $r < r_1$ was arbitrary, Montel's theorem implies that after passing to a subsequence, $\psi_{k}$ converges to a holomorphic map $\psi:\Bb_d(0,r_1) \rightarrow \Omega_{\infty}$.

We next claim that $\psi$ is locally invertible at $z=0$. It follows from Condition (3) in Definition~\ref{defn:quasi_bd_geom} that 
$$
\frac{1}{C} \norm{v} \leq \sqrt{g_{n_k}(d(\psi_k)_0 v,d(\psi_k)_0 v)}
$$  
for every $k \geq 1$ and $v \in \Cb^d$. Then by Lemma~\ref{lem:basic_est} we have
$$
\frac{1}{C} \norm{v} \leq \sqrt{g_{n_k}(d(\psi_k)_0 v,d(\psi_k)_0 v)} \leq A  k_{\Omega_{n_k}}(z_k;d(\psi_k)_0 v) \leq  A \frac{\norm{d(\psi_k)_0 v}}{\delta_{\Omega_{n_k}}(z_k)}.
$$
In particular,
$$
\norm{d(\psi)_0 v} = \lim_{k \rightarrow \infty} \norm{d(\psi_k)_0 v} \geq \lim_{k \rightarrow \infty} \frac{1}{AC} \norm{v}\delta_{\Omega_{n_k}}(z_k) = \frac{\delta_{\Omega_\infty}(z_\infty)}{AC} \norm{v}.
$$
This implies that $\psi$ is non-singular at 0 and hence locally invertible. 

Pick a neighborhood $U_1$ of $0$ such that $\psi|_{U_1}$ is invertible. Next fix a neighborhood $U_2$ of $0$ such that $\overline{U_2} \subset U_1$. Since $\psi_k$ converges in the $C^\infty$ topology to $\psi$, we can find $M >0$ such that $\psi_k|_{U_2}$ is invertible when $k \geq M$. Then $\psi_k|_{U_2}^{-1}$ converges locally uniformly to $\psi|_{U_2}^{-1}$.

Next fix a neighborhood $V$ of $z_\infty = \psi(0)$ such that $\overline{V} \subset \psi(U_2)$. By increasing $M$, we can assume that $V \subset \psi_k(U_2)$ for all $k \geq M$ and
\begin{align}
\sup_{k \geq M} \sup_{z \in V} \abs{ \frac{ \partial^{\abs{a}+\abs{b}} \left(\psi_k|_{U_2}\right)^{-1}}{\partial z^a \partial \overline{z}^{b}}(z)} < \infty \label{eq:estimate}
\end{align}
for every $\abs{a}+\abs{b} \leq \abs{\mu}+\abs{\nu}$. By possibly increasing $M$ again we can assume that $z_k \in V$ for all $k \geq M$.  But then Condition~\eqref{eq:estimate} and Condition (4) in Definition~\ref{defn:quasi_bd_geom} imply that 
\begin{align*}
\sup_{k \geq M} \abs{ \frac{ \partial^{\abs{\mu}+\abs{\nu}} (g_{n_k})_{i\overline{j}}}{\partial z^\mu \partial \overline{z}^{\nu}}(z_k)} < \infty,
\end{align*}
which is a contradiction.

Hence, after passing to a subsequence, we may assume that $g_n$ converges locally uniformly in the $C^\infty$ topology to an indefinite Hermitian metric $g_{\infty}$ on $\Omega_{\infty}$. It remains to show that $g_{\infty}$ is positive definite and K\"ahler.
From Theorem~\ref{conv-thm} we have
$$
\begin{array}{lll}
\sqrt{g_{\infty}(z)(v,v)} =\lim_{k \rightarrow \infty}\sqrt{g_{n_k}(z)(v,v)}  & \geq & \frac{1}{A} \lim_{k \rightarrow \infty} k_{\Omega_{n_k}}(z;v)\\
 & = & \frac{1}{A} k_{\Omega_{\infty}}(z;v)\\
 & > & 0
\end{array}
$$
for every $z \in \Omega_{\infty}$ and $v \in \Cb^d$. Hence $g_{\infty}$ is positive definite. Finally, since each $g_n$ is a K\"ahler metric and the convergence is locally uniformly in the $C^{\infty}$ topology, the metric $g_{\infty}$ is also K\"ahler.
\end{proof}

\section{Ricci flow on K{\"a}hler manifolds}\label{section:ricci}

In this section we use the Ricci flow to deform a complete K\"ahler metric $g$ with bounded sectional curvature and obtain a new K{\"a}hler metric with better properties. 

\begin{theorem}\label{kap-thm}Let $(M,J)$ be a complex manifold of dimension $d$ with complete K{\"a}hler metric $g$ and let $T(g)$ denote either $H(g)$ or $B(g)$. Suppose that the sectional curvature of $g$ is bounded in absolute value by a real number $\kappa> 0$. Then for every $\epsilon > 0$ and $r > 0$ there exists a complete K{\"a}hler metric $h$ with the following properties:
\begin{enumerate}
\item[($i$)] $g$ and $h$ are $(1+\epsilon)$-bi-Lipschitz,
\item[($ii$)] for every $q \geq 0$ there exists $C_q > 0$ such that 
\begin{align*}
\sup_{M} \norm{\nabla^q R(h)}_{h} \leq C_q,
\end{align*}
\item[($iii$)] for every $z \in M$ 
\begin{align*}
\inf_{B_g(z,r)} T(g)-\epsilon \leq T(h)|_{B_g(z,r)} \leq \sup_{B_g(z,r)} T(g)+\epsilon
\end{align*}
where $B_g(z,r)$ is the ball of radius $r$ centered at $z$ in the distance induced by $g$. 
\end{enumerate}
Moreover, for every $q \geq 0$, the constant $C_q$ can be chosen to depend only on $q$, $\kappa$, $\epsilon$, $r$, and $d$.
\end{theorem}

\begin{remark} Everything but Part ($iii$) follows from results of W. X. Shi~\cite{Shi89, Shi97}. To prove Part ($iii$) we adapt an argument of V. Kapovich~\cite{Kap05}. A global version of Part ($iii$) for sectional curvature and holomorphic sectional curvature appears in~\cite[Lemma 13]{WY17}. \end{remark}

\begin{proof}
We assume, without loss of generality, that the sectional curvatures of $g$ are uniformly bounded on $M$ between -1 and 1, namely $\kappa = 1$. We recall that the Ricci flow of $g$ is given by 
\begin{equation}\label{ricci-eq}
\frac{\partial}{\partial t}g=-2 { \rm Ric}(g)
\end{equation}
where ${\rm Ric}(g)$ denotes the Ricci curvature tensor of $g$. 

By Theorem 1.1 in~\cite{Shi89}, Equation (\ref{ricci-eq}) has some solution $g_t$ for every $t \in [0,t_0]$, where $t_0 > 0$ depends only on $d$ and $\kappa$. Moreover, there exists $c(d,t_0) > 0$ and, for every $q \geq 0$, there exists $c(d,q,t_0) > 0$ such that $g_t$ satisfies the following conditions \begin{equation}\label{estimates-eq}
e^{-c(d,t_0)t} g \leq g_t \leq e^{c(d,t_0)t}g \ \ \ {\rm and} \ \ \ \norm{\bigtriangledown^qR(g_t)}_{g_t}^2 \leq \frac{c(d,q,t_0)}{t^q}.
\end{equation}
By Theorem 5.1 in~\cite{Shi97}, $g_t$ is a K\"ahler metric for any $t \in [0,t_0]$.

V. Kapovich proved in~\cite[Proposition, Remark 1]{Kap05} that for every $r > 0$ there exists a constant $C(d,r,t_0)>0$ such that 
\begin{align*}
\inf_{B_g(z,r)} K(g) -C(d,R,t_0)t \leq K(g_t)|_{B_g(z,r)} \leq \sup_{B_g(z,r)} K(g)+C(d,r,t_0)t
\end{align*}
where $K$ denotes the sectional curvature. His argument can also be used to show that 
\begin{align}\label{eq:local_est}
\inf_{B_g(z,r)} T(g) -C(d,R,t_0)t \leq T(g_t)|_{B_g(z,r)} \leq \sup_{B_g(z,r)} T(g)+C(d,r,t_0)t
\end{align}
(after possibly increasing $C(d,r,t_0)$). We now explain the necessary modifications. 

Fix $U,V \in T_{x_0} M$ with $\norm{U}_g = \norm{V}_g = 1$. The change consists in replacing $\Phi_z(x,t)$ in the proof of Proposition in \cite{Kap05} with the function
$$
\wt{\Phi}_z(x,t):= B(g_t)(x,U,V) \zeta_z(x)=\frac{R(g_t)(U,JU,V,JV)}{|U|^2_{g_t}|V|^2_{g_t}}\zeta_z(x),
$$
when $T=B$, and with the function

$$
\wt{\Phi}_z(x,t):= H(g_t)(x,U) \zeta_z(x)=\frac{R(g_t)(U,JU,JU,U)}{|U|^4_{g_t}}\zeta_z(x)
$$
when $T=H$.

Then the proof follows line by line the proof of Proposition in \cite{Kap05}, replacing everywhere $|U \wedge V|$ with $|U|_{g_t} \cdot |V|_{g_t}$ when $T=H$ and  with $|U|^2_{g_t}$ when $T=B$. This modification also requires the fact that there exists a constant $C(d,t_0) > 0$ such that for every $x \in M$ and for every $t \in [0,t_0]$,
$$
\left| \frac{\partial |U|_{g_t}(x,t)|}{\partial t}\right| \leq C(d,t_0), \ \left| \frac{\partial |V|_{g_t}(x,t)|}{\partial t}\right| \leq C(d,t_0).
$$

It follows from Equations~\eqref{estimates-eq} and~\eqref{eq:local_est} that we can pick some $t > 0$, which only depends on $\kappa$, $\epsilon$, $R$, and $d$, such that the metric $h = g_t$ satisfies Parts ($i$), ($ii$), and ($iii$) of Theorem~\ref{kap-thm}. Moreover, for every $q$ the constant $C_q$ can be chosen to only depend on $q$, $\kappa$, $\epsilon$, $r$, $d$, and are provided by (\ref{estimates-eq}). This ends the proof.
\end{proof}

\section{Metric deformation on convex domains: Proof of Theorem~\ref{thm:affine}}\label{sec:pf_of_thm_affine}

In this section we prove the following stronger version of Theorem~\ref{thm:affine}.

\begin{theorem}\label{thm:affine_closures} Suppose that $\Omega \in \Xb_d$, $g$ is a complete K{\"a}hler metric on $\Omega$, and $T(g)$ is either $H(g)$ or $B(g)$. Assume there exist $0 < b < a$ such that
\begin{align*}
-a \leq T(g)  \leq -b
\end{align*}
outside a compact subset $K$ of $\Omega$. If $\Omega_\infty \in { \rm AL}(\Omega)$ and $\epsilon > 0$, then there exists a complete K{\"a}hler metric $g_\infty$ on $\Omega_\infty$ with
\begin{enumerate}
\item $-b-\epsilon \leq T(g_\infty)  \leq -a+\epsilon$ on $\Omega_\infty$,
\item there exists $A_{\infty}>1$ such that $g_\infty$ and $k_{\Omega_\infty}$ are $A_{\infty}$-bi-Lipschitz on $\Omega_\infty$, and
\item for every $q \geq 0$ there exists $C_q > 0$ such that 
\begin{align*}
\sup_{\Omega_\infty} \norm{\nabla^q R(g_{\infty})}_{g_\infty} \leq C_q.
\end{align*}
\end{enumerate}
Moreover, the constant $A_{\infty}$ can be chosen to depend only on $a$, $b$, $\epsilon$, and $d$ and, for every $q \geq 0$, the constant $C_q$ can be chosen to depend only on $q$, $a$, $b$, $\epsilon$, and $d$. 

\end{theorem}

We first prove the following lemma.

\begin{lemma}\label{kob-lem1} Under the hypothesis of Theorem~\ref{thm:affine_closures}:
\begin{itemize}
\item[(i)] There exist $A > 1$ (depending only on $a$ and $b$)  and a compact subset $K'$ of $\Omega$ such that 
\begin{equation*}
\sqrt{g(z)(v,v)} \leq A k_{\Omega}(z;v) 
\end{equation*}
for all $z \in \Omega \backslash K'$ and $v \in \Cb^d$.
\item[(ii)] There exists $A'> A$ such that 
\begin{equation*}
\frac{1}{A'} k_{\Omega}(z;v) \leq \sqrt{g(z)(v,v)} \leq A' k_{\Omega}(z;v) 
\end{equation*}
for all $z \in \Omega$ and $v \in \Cb^d$.
\end{itemize}
\end{lemma}

\begin{proof}
First, by the assumptions on $T(g)$ and according to the Yau-Schwarz Lemma (see \cite{Yau78}, Theorem 2), there is a constant $A_1 > 0$, depending only on $a$ and $b$, such that 
\begin{align*}
\sqrt{g(z)(v,v)} \leq A_1 k_{\Omega \backslash K}(z;v)
\end{align*} 
for all $z \in \Omega \backslash K$ and $v \in \mathbb C^n$.

Since $K_\Omega$ is a proper distance on $\Omega$, we can fix a compact set $K^\prime \subset \Omega$ such that $K \subset K^\prime$ and 
$$
\min_{z \in \Omega \setminus K^\prime} K_\Omega(z, K) > K_{\Db}\left(0,\frac{1}{2} \right).
$$
Notice that if $\varphi : \Db \rightarrow \Omega$ is holomorphic and $\varphi(0) \in \Omega \setminus K^\prime$, then $\varphi\left( \frac{1}{2} \cdot \Db\right) \subset \Omega \setminus K$. Hence 
$$
k_{\Omega \backslash K}(z;v) \leq 2 k_{\Omega}(z;v)
$$
for all $z \in \Omega \setminus K^\prime$ and $v \in \Cb$. Consequently, 
\begin{align*}
\sqrt{g(z)(v,v)} \leq 2 A_1 k_{\Omega}(z;v)
\end{align*}
for every $z \in \Omega \backslash K'$ and every $v \in \mathbb C^d$.
This proves Part (i).

By compactness, there is a positive constant $A_2$ such that $\sqrt{g(z)(v,v)} \leq A_2 k_{\Omega}(z;v)$ for every $z \in K'$ and every $v \in \mathbb C^d$.
Hence
$$
\sqrt{g(z)(v,v)} \leq \max\{2A_1,A_2\} k_{\Omega}(z;v)
$$
for every $z \in \Omega$ and every $v \in \mathbb C^d$. This proves the upper estimate of Part (ii).

It follows now from the assumptions on $T(g)$ and from the smoothness of the complete K\"ahler metric $g$ on $\Omega$ that the Ricci curvature of $g$ is bounded from below and above on $\Omega$ according to~\cite[Formula (6.1)]{BG64}. Hence, again from the Yau-Schwarz Lemma~\cite{Yau78}, we obtain that there exists $A_3 > 0$ such that
$$
\sqrt{g(z)(v,v)} \geq \frac{1}{A_3} c_{\Omega}(z;v)
$$
for every $z \in \Omega$ and every $v \in \mathbb C^d$. Here $c_{\Omega}$ denotes the Carath\'eodory infinitesimal metric on $\Omega$. Finally, since $\Omega$ is convex, $c_{\Omega} \equiv k_{\Omega}$. This completes the proof of Part (ii), setting $A':=\max\{2A_1, A_2, A_3\}$.

\end{proof}

We can prove now Theorem~\ref{thm:affine_closures}.

\begin{proof}[Proof of Theorem~\ref{thm:affine_closures}]
There exists $k > 0$  (depending only on $a$, $b$, and $d$) such that the curvature tensor $R(g)$ of $(\Omega, g)$ satisfies
\begin{equation}\label{ric-eq}
\norm{R(g)(z)}_{g} \leq k
\end{equation}
for every $z \in \Omega \backslash K$ (see~\cite[Formula (6.1)]{BG64}). Then, since $K$ is compact, there exists some $\wt{k}> k$  such that 
\begin{equation}\label{ric-eq2}
\norm{R(g)(z)}_g \leq \wt{k}
\end{equation}
for every $z \in \Omega$.

Fix $\epsilon > 0$. By Theorem~\ref{kap-thm}, there exists a complete K{\"a}hler metric $h$ on $\Omega$ such that 
\begin{enumerate}
\item[($i$)] $g$ and $h$ are $(1+\epsilon/2)$-bi-Lipschitz,
\item[($ii$)] for every $q \geq 0$ there exists $\wt{C}_q > 0$ such that 
\begin{align*}
\sup_{\Omega} \norm{\nabla^q R(h)}_{h} \leq \wt{C}_q,
\end{align*}
\item[($iii$)] $a-\epsilon/2 \leq T(h) \leq b+\epsilon/2$  on $\Omega \backslash K$.
\end{enumerate}

Next fix $\Omega_\infty \in { \rm AL}(\Omega)$. By definition, there exist a sequence of points $z_n \in \Omega$, a point $z_\infty \in \Omega_\infty$, and affine maps $A_n \in \Aff(\Cb^d)$, such that 
\begin{enumerate}
\item $z_n \rightarrow q \in \partial \Omega$,
\item $A_n(\Omega, z_n)$ converges to $(\Omega_\infty, z_\infty)$. 
\end{enumerate}

Let $\Omega_n = A_n \Omega$ and $h_n = (A_n)^*h$. Since the Kobayashi metric is invariant under biholomorphisms, it follows from Property ($i$) of $h$ and Lemma~\ref{kob-lem1} Part (ii)  that 
\begin{align*}
\frac{1}{A' (1+\varepsilon/2)}k_{\Omega_n}(z;v) \leq \sqrt{h_n(z)(v,v)} \leq A' (1+\varepsilon/2)k_{\Omega_n}(z;v),
\end{align*}
for every $n \geq 1$, $z \in \Omega_n$, and $v \in \Cb^d$. We emphasize that the constant $A'$ depends on $\Omega$ but not on $n$.

Further, for every $n$ the curvature tensor $R(h_n)$ of $(\Omega_n, h_n)$ satisfies 
\begin{align*}
\sup_{\Omega_n} \norm{\nabla^q R(h_n)}_{h_n} =\sup_{\Omega} \norm{\nabla^q R(h)}_{h} \leq \wt{C}_q.
\end{align*}

Hence, by Proposition~\ref{prop:rescaling_metrics} we can pass to a subsequence and assume that $h_n$ converges locally uniformly in the $C^{\infty}$ topology to some complete K\"ahler metric $h_\infty$ on $\Omega_{\infty}$. Moreover, by construction and Property ($iii$) of $h$, we have 
\begin{equation}\label{hol-eq}
-a - \epsilon/2 \leq T(h_\infty) \leq -b +  \epsilon/2
\end{equation}
on $\Omega_\infty$.

We obtain from~\eqref{hol-eq} that the sectional curvatures of $h_{\infty}$ are bounded on $\Omega_{\infty}$ between $-\kappa_{\infty}$ and $\kappa_{\infty}$, where $\kappa_{\infty}$ is a positive constant depending only on $a$, $b$, $d$, and $\epsilon$ (see~\cite[Formula (6.1)]{BG64}).

Now, since $T(h_{\infty})$ is negatively pinched on $\Omega_{\infty}$ by (\ref{hol-eq}), it follows from Lemma~\ref{kob-lem1} (i) that there exists $\beta_1 > 1$, depending only on $a$, $b$, $d$, and $\epsilon$, such that
\begin{equation*}
\sqrt{h_{\infty}(z)(v,v)} \leq \beta_1 k_{\Omega_{\infty}}(z;v)
\end{equation*}
for every $z \in \Omega_{\infty}$ and $v \in \Cb^d$. Moreover, repeating the proof of the lower estimate of Lemma~\ref{kob-lem1} (ii), there exists $\beta_2 > 0$, depending only on $a$, $b$, $d$,
$\epsilon$, and $\kappa_{\infty}$ (and consequently only on $a$, $b$, $d$, and $\epsilon$) such that
\begin{equation*}
\frac{1}{\beta_2} k_{\Omega_{\infty}}(z;v) \leq \sqrt{h_{\infty}(z)(v,v)}
\end{equation*}
for every $z \in \Omega_{\infty}$ and $v \in \Cb^d$. In particular, setting $\beta:=\max\{\beta_1,\beta_2\}$, the metrics $k_{\Omega_{\infty}}$ and $h_{\infty}$ are $\beta$-bi-Lipschitz on $\Omega_{\infty}$.

Finally, applying Theorem~\ref{kap-thm} to $(\Omega_{\infty},h_\infty)$, with $\varepsilon/2$ instead of $\varepsilon$, we obtain that there exists a complete K\"ahler metric $g_{\infty}$ on $\Omega_\infty$ which satisfies the conclusion of Theorem~\ref{thm:affine_closures}, with $A_{\infty}=\beta(1+\varepsilon)$.

\end{proof}

\section{Proof of Theorem~\ref{type-thm}}\label{sec:pf_of_main_thm}

The final ingredient needed to prove Theorems~\ref{type-thm} is Proposition~\ref{prop:no_embedded_products}. This is a consequence of a result of F. Zheng~\cite{Zhe94}. Before stating this result we need one definition. 

\begin{definition}\cite{Zhe94} Let $\Omega$ be a bounded pseudoconvex domain in $\Cb^d$ and let $g_\Omega$ denote the unique complete  K\"ahler-Einstein  metric on $\Omega$ with Ricci curvature $-(d+1)$. Then $\Omega$ has \emph{geometric rank $\geq$ 2} if there is a complete K\"ahler manifold $(M, g_0)$ with Ricci curvature bounded from below and a holomorphic embedding $f:\mathbb D \times M \rightarrow \Omega$ such that $f_t^*(g_\Omega) \geq g_0$ for every $t \in \mathbb D$, where $f_t=f(t,\cdot)$.\end{definition}

\begin{theorem}[\cite{Zhe94}, Theorem A]\label{thm:zheng} Suppose that $\Omega$ is a bounded pseudoconvex domain. If $\Omega$ has geometric rank $\geq$ 2, then it does not admit a complete K\"ahler metric with negatively pinched holomorphic bisectional curvature.
\end{theorem}

Every domain $\Omega \in \Xb_d$ is biholomorphic to a bounded pseudoconvex domain (see for instance~\cite[Proposition 2.8]{Fra87}) and hence has a unique complete K\"ahler-Einstein metric with Ricci curvature $-(d+1)$ which we denote by $g_\Omega$. In the proof of Proposition~\ref{prop:no_embedded_products}, we will require the following estimates on $g_\Omega$ and $k_\Omega$.

\begin{lemma}\label{lem:metric_comp} \cite[Theorem 2.2]{Fra91} For any $d > 0$ there exists $C>1$ such that: if $\Omega \in \Xb_d$, $z \in \Omega$, and $v \in \mathbb C^n$, then 
\begin{align*}
\frac{1}{C} k_\Omega(z;v)^2 \leq g_\Omega(z)(v,v) \leq C k_\Omega(z;v)^2.
\end{align*}
\end{lemma}

Lemma~\ref{lem:metric_comp} also follows from general results about the squeezing function, see~\cite[Theorem 2]{Yeung2009} and~\cite[Theorem 1]{NA2017}.

\begin{lemma} If $\Omega \in \Xb_d$ and $V$ is a complex affine $k$-plane, then 
\begin{align*}
k_\Omega(z;v) \leq k_{\Omega \cap V}(z;v) \leq 2 k_\Omega(z;v)
\end{align*}
for all $z \in V \cap \Omega$ and $v \in T_z V$. 
\end{lemma}

\begin{proof} Since $\Omega \cap V \subset \Omega$, the definition of the Kobayashi metric implies that 
\begin{align*}
k_\Omega(z;v) \leq k_{\Omega \cap V}(z;v) 
\end{align*}
for all $z \in V \cap \Omega$ and $v \in T_z V$. Thus by Lemma~\ref{lem:basic_est} we have
\begin{align*}
k_{\Omega \cap V}(z;v) \leq \frac{\norm{v}}{\delta_{\Omega \cap V}(z;v)} =\frac{\norm{v}}{\delta_{\Omega}(z;v)} \leq 2 k_\Omega(z;v)
\end{align*}
for all $z \in V \cap \Omega$ and $v \in T_z V$, $v \neq 0$. 
\end{proof}

We can now prove Proposition~\ref{prop:no_embedded_products}.

\begin{proof}[Proof of Proposition~\ref{prop:no_embedded_products}]
Recall, that $k_{\Db}$ is a scalar multiple of the Poincar\'e metric on $\Db$ and so 
$$
h(z_1,z_2)\Big( (v_1, v_2), (v_1, v_2) \Big) = k_{\Db}(z_1;v_1)^2 + k_{\Db}(z_2;v_2)^2
$$
defines a K\"ahler metric on $\Db \times \Db$. Also, the Kobayashi metric on $\Db \times \Db$ satisfies
$$
k_{\Db \times \Db}\Big( (z_1, z_2); (v_1, v_2) \Big) = \max\left\{ k_{\Db}(z_1;v_1), k_{\Db}(z_2; v_2)\right\}.
$$

By assumption, there is a biholomorphism $\psi:\mathbb D \times \mathbb D \rightarrow V \cap \Omega$. If $\iota:V \cap\Omega\rightarrow \Omega$ denotes the inclusion map, then $f:=\iota \circ \psi$ is a holomorphic embedding of $\mathbb D \times \mathbb D$ into $\Omega$. 

Let $C > 1$ be the constant from Lemma~\ref{lem:metric_comp}. Then for every $z=(z_1,z_2) \in \Db \times \Db$ and $v=(v_1,v_2) \in \Cb \times \Cb$, we have
\begin{align*}
(f^*g_\Omega)(z)(v,v)
&=  g_{\Omega}(f(z))(d(f)_zv, d(f)_zv) \geq \frac{1}{C}k_{\Omega}(f(z); d(f)_zv)^2 \\
& \geq \frac{1}{4C} k_{\Omega \cap V}(f(z); d(f)_zv)^2 = \frac{1}{4C} k_{\Db \times \Db}(z;v)^2 \\
& = \frac{1}{4C} \max\left\{ k_{\Db}(z_1;v_1)^2, k_{\Db}(z_2; v_2)^2\right\} \geq \frac{1}{8C} h(z)(v,v).
\end{align*}
So $\Omega$ has geometric rank $\geq 2$ and hence by Theorem~\ref{thm:zheng} does not admit a complete K{\"a}hler metric with pinched negative holomorphic bisectional curvature.
\end{proof}

\begin{proof}[Proof of Theorem~\ref{type-thm}] Suppose that $\Omega \subset \Cb^d$ is a convex domain and $g$ is a complete K{\"a}hler metric on $\Omega$ with pinched negative holomorphic bisectional curvature outside a compact set $K \subset \Omega$.

We first show that $\Omega$ does not contain any complex affine lines, i.e. $\Omega \in \Xb_d$. Suppose for a contradiction that there exists $a,b \in \Cb^d$ with $b \neq 0$ and 
\begin{align*}
a + \Cb \cdot b\subset \Omega.
\end{align*}
Since $\Omega$ is convex and open, this implies that $z + \Cb \cdot b \subset \Omega$ for every $z \in \Omega$. So by applying an affine transformation to $\Omega$, we can assume that $\Omega = \Cb \times \Omega^\prime$ for some convex domain $\Omega^\prime \subset \Cb^{d-1}$. Repeating the argument at the start of the proof of Lemma~\ref{kob-lem1}, there exists some some $A_1 > 0$ such that \begin{align*}
\sqrt{g(z)(v,v)} \leq A_1 k_{\Omega \backslash K}(z;v)
\end{align*} for every $z \in \Omega \backslash K$ and every $v \in \mathbb C^n$. 

Now let $v_0 = (1,0,\dots,0)$ and pick some $z_0 \in \Omega$ such that 
\begin{align*}
z_0 + \Cb \cdot v_0 \subset \Omega \setminus K
\end{align*}
(this is possible since $\Omega = \Cb \times \Omega^\prime$). Then 
\begin{align*}
0 < \sqrt{g(z_0)(v_0,v_0)} \leq A_1 k_{\Omega \setminus K}(z_0;v_0) = 0
\end{align*}
which is a contradiction. Thus $\Omega \in \Xb_d$.

Next, suppose for a contradiction that either there is a nontrivial holomorphic map from $\mathbb D$ to $\partial \Omega$ or $\partial \Omega$ is $C^{\infty}$ smooth and has a point of infinite type. Then according to Theorem~\ref{thm:construcing_embedded_products}, there exists some $\Omega_{\infty} \in { \rm AL}(\Omega)$ and a complex affine 2-plane $V$ such that $V \cap \Omega_\infty$ is biholomorphic to $\Db \times \Db$.

It follows from Proposition~\ref{prop:no_embedded_products} that $\Omega_{\infty}$ does not admit a complete K{\"a}hler metric with pinched negative holomorphic bisectional curvature. Hence, according to Theorem~\ref{thm:affine}, $\Omega$ does not admit a complete K{\"a}hler metric with pinched negative holomorphic bisectional curvature outside a compact set. So we have a contradiction.
\end{proof}

\section{Proof of Theorem~\ref{thm:char_str_pconvex}}\label{sec:char_str_convex}

For the proof of Theorem~\ref{thm:char_str_pconvex}, we will use the following theorem of R. Greene and S. Krantz.

\begin{theorem}\label{thm:GK}\cite[Theorem 3]{GK1981} 
Suppose that $M$ is a simply connected complex manifold of complex dimension $d$ and for every $\epsilon > 0$ there is a complete K{\"a}hler metric $g$ on $M$ with 
\begin{align*}
-1 -\epsilon \leq H(g) \leq -1+\epsilon.
\end{align*}
Then $M$ is biholomorphic to the unit ball in $\Cb^d$. 
\end{theorem}

Using results from~\cite{Zim18} we will establish the following. 

\begin{theorem}\label{thm:compact_set} For every $\alpha \in (0,1)$ and $d > 0$, there exists a subset $\mathbb{L}_{d,\alpha} \subset \Xb_d$ with the following properties:
\begin{enumerate}
\item $\mathbb{L}_{d,\alpha}$ is compact in $\Xb_d$, 
\item if $\Cc \in \mathbb{L}_{d,\alpha}$, then $\Cc$ is not biholomorphic to the unit ball, and 
\item if $\Omega \subset \Cb^d$ is a bounded convex domain with $C^{2,\alpha}$ boundary which is not strongly pseudoconvex, then there exists a domain $\Cc \in { \rm AL}(\Omega) \cap \mathbb{L}_{d,\alpha}$. 
\end{enumerate}
\end{theorem}

\begin{proof} 

Let $\Kb_d$ be the set defined in Proposition~\ref{prop:compact_set}. Next, let $\mathbb{L}_{d,\alpha}$ be the set of all convex domains $\Cc \in \Xb_d$ such that
\begin{enumerate}
\item $\Cc \in \Kb_d$,
\item $(\Cb \times\{0\}) \cap \Cc = \{ (z,0,\dots,0) : { \rm Im}(z) > 0\}$, and
\item $\delta_{\Cc}(rie_1; e_2) \leq r^{1/(2+\alpha)}$ for $r \geq 1$.
\end{enumerate}
Conditions (2) and (3) are clearly closed conditions in the local Hausdorff topology. Then, since $\Kb_d$ is compact in $\Xb_d$, we see that $\mathbb{L}_{d,\alpha}$ is compact in $\Xb_d$.

By~\cite[Proposition 5.1]{Zim18},  if $\Omega \subset \Cb^d$ is a bounded convex domain with $C^{2,\alpha}$ boundary which is not strongly pseudoconvex, then there exists a domain $\Cc \in { \rm AL}(\Omega) \cap \mathbb{L}_{d,\alpha}$. By~\cite[Proposition 2.1]{Zim18}, if $\Cc \in \mathbb{L}_{d,\alpha}$, then $\Cc$ is not biholomorphic to the unit ball.
\end{proof}

\begin{proof}[Proof of Theorem~\ref{thm:char_str_pconvex}] By scaling it suffices to consider the case when $c=1$. Fix $\alpha \in (0,1)$ and $d \geq 2$. Using Theorem~\ref{thm:curv}, it is enough to prove the following: there exists some $\epsilon = \epsilon(\alpha, d) >0$ such that if $\Omega \subset \Cb^d$ is a bounded convex domain with $C^{2,\alpha}$ boundary and there exists a complete K{\"a}hler metric $g$ on $\Omega$ with
\begin{align*}
-1-\epsilon \leq H(g) \leq -1 + \epsilon
\end{align*}
outside a (possibly empty) compact subset of $\Omega$, then $\Omega$ is strongly pseudoconvex. 

Suppose for a contradiction that this statement is false. Then for every $n \in \Nb$ there exist $\Omega_n \subset \Cb^d$ a bounded convex domain with $C^{2,\alpha}$ boundary which is not strongly pseudoconvex, a compact set $K_n \subset \Omega_n$, and a complete K{\"a}hler metric $g_n$ on $\Omega_n$ with
\begin{align*}
-1-1/n \leq H(g_n) \leq -1 + 1/n
\end{align*}
on $\Omega_n \setminus K_n$.

By Theorem~\ref{thm:compact_set}, there exists some $\Cc_n \in { \rm AL}(\Omega_n) \cap \mathbb{L}_{d,\alpha}$. Since $\mathbb{L}_{d,\alpha} \subset \Xb_d$ is a compact set, we can pass to a subsequence and assume that $\Cc_n$ converges to some $\Cc_\infty$ in $\mathbb{L}_{d,\alpha}$. By Theorem~\ref{thm:compact_set}, $\Cc_\infty$ is not biholomorphic to the unit ball. Then Theorem~\ref{thm:GK} and the following Claim give a contradiction. \\

\noindent \textbf{Claim.} For every $\delta> 0$, there is a complete K{\"a}hler metric $h$ on $\Cc_\infty$ such that 
\begin{align*}
-1 -\delta \leq H(h) \leq -1+\delta.
\end{align*}

\noindent \emph{Proof of Claim:} Fix $\delta > 0$. For $n \geq 2/\delta$, we have 
\begin{align*}
-1-\delta/2 \leq H(g_n)  \leq -1+\delta/2
\end{align*}
in a neighborhood of $\partial \Omega_n$. So Theorem~\ref{thm:affine_closures} with $\epsilon = \delta/2$ implies that  there exist constants $A>1$ and $\{ C_q\}_{q \geq 0} $ such that for every $n \geq 2/\delta$ there exists a complete K{\"a}hler metric $h_n$ on $\Cc_n$ with
\begin{enumerate}
\item $-1-\delta \leq H(h_n)  \leq -1+\delta$ on all of $\Cc_n$,
\item $h_n$ and $k_{\Cc_n}$ are $A$-bi-Lipschitz on $\Cc_n$, and
\item 
\begin{align*}
\sup_{z \in \Cc_n} \norm{\nabla^q R(h_n) }_{h_n} \leq C_q.
\end{align*}
\end{enumerate}

Using Proposition~\ref{prop:rescaling_metrics}  and possibly passing to a subsequence, we can assume that $h_n$ converges locally uniformly in the $C^\infty$ topology to a K{\"a}hler metric $h$ on $\Cc_\infty$. Then 
\begin{align*}
-1 -\delta \leq H(h) \leq -1+\delta
\end{align*}
and the Claim is established. 

\end{proof}

\appendix

\section{Proof of Theorem~\ref{conv-thm}}\label{app:pf_of_conv-thm}

For the rest of the section suppose that $\Omega_n$ converges to $\Omega$ in $\Xb_d$. We first establish the following lemma.

\begin{lemma}\label{lem:compact_inclusion} For every compact set $K \subset \Omega$, there exists $N \geq 0$ such that $K \subset \Omega_n$ for all $n \geq N$. \end{lemma}

\begin{proof} The proof follows directly from the definition of convergence in $\Xb_d$ and convexity. Indeed, fix a compact set $K \subset \Omega$ and suppose that there does not exist some $N \geq 0$ with $K \subset \Omega_n$ for all $n \geq N$. Then there exist $n_j \rightarrow \infty$ and $k_j \in K \setminus \Omega_{n_j}$. By passing to a subsequence we can assume that $k_j \rightarrow k \in K$. 

Since $\Omega_{n_j}$ is convex, there exists some affine map $\ell_j : \Cb^d \rightarrow\Cb$ such that $\ell_j(k_j)=0$ and ${ \rm Im}( \ell_j(\Omega_{n_j})) > 0$. Next pick $a_j \in \Cb$ and $b_j \in \Cb^d$ such that $\ell_j(z)= a_j + \ip{b_j,z}$. We can assume that $\norm{b_j}=1$ and then 
\begin{align*}
\abs{a_j} = \abs{\ip{b_j, k_j}} \leq \norm{k_j}.
\end{align*}
So by passing to a subsequence we can suppose that $\ell_j$ converges to a non-constant affine map $\ell: \Cb^d \rightarrow \Cb$. If $z \in \Omega$, then there exists $z_n \in \Omega_n$ such that $z_n \rightarrow z$. So 
\begin{align*}
{ \rm Im} (\ell(z)) = \lim_{j \rightarrow \infty} { \rm Im}(\ell_j(z_{n_j})) \geq 0.
\end{align*}
Since $z \in \Omega$ was arbitrary we have
\begin{align*}
{ \rm Im}\left( \ell\left(\Omega\right)\right) \geq 0.
\end{align*}
Since $\Omega$ is open we must have ${ \rm Im}\left( \ell\left(\Omega\right)\right) > 0$. But then 
\begin{align*}
0 < { \rm Im}(\ell(k)) = \lim_{j \rightarrow \infty} { \rm Im}(\ell_j(k_j)) = \lim_{j \rightarrow \infty} 0 = 0
\end{align*}
which is a contradiction. 
\end{proof}

Let $S = \{ v \in \Cb^d : \norm{v}=1\}$. To prove Theorem~\ref{conv-thm}, it is enough to prove the convergence of the Kobayashi metrics $k_{\Omega_n}$ on compact subsets of $\Omega \times S$. Fix a compact subset  $K \subset \Omega$. Then since $K \times S$ is compact, it is enough to consider a sequence $(p_n, v_n) \in K  \times S$ with 
\begin{align*}
\lim_{n \rightarrow \infty} (p_n, v_n) = (p,v)
\end{align*}
and show that 
$$
\lim_{n \rightarrow \infty}k_{\Omega_n}(p_n;v_n) = k_{\Omega}(p;v). 
$$
Notice that Lemma~\ref{lem:compact_inclusion} implies that $p_n \in \Omega_n$ for $n$ sufficiently large and hence $k_{\Omega_n}(p_n; v_n)$ is well defined for $n$ sufficiently large. 

\begin{lemma}\label{2-lem}
\begin{align*}
\limsup_{n \rightarrow \infty}k_{\Omega_n}(p_n;v_n) \leq k_\Omega(p;v).
\end{align*}
\end{lemma}

\begin{proof}
Fix some $r \in (0,1)$ and let $D_r:=\{\zeta \in \Cb : \ |\zeta| < r\}$. Then the set 
\begin{align*}
\wh{K} = \left\{ g(\zeta) : g : \Db \rightarrow \Omega \text{ holomorphic, } g(0) \in K, \text{ and } \zeta \in \overline{D_r} \right\}
\end{align*}
is compact in $\Omega$ since the Kobayashi distance is proper. 

For every $n$, let $g_n:\mathbb D \rightarrow \Omega$ be a holomorphic map such that $g_n(0) = p_n$, $g_n'(0) = v_n /\alpha_n$, and 
$$
\alpha_n = k_{\Omega}(p_n,v_n).
$$
Since $g_n(D_r) \subset \wh{K}$, there exists some $N_r \geq 0$ such that $g_n(D_r) \subset \Omega_n$ for all $n \geq N_r$. Then define  $g_{n,r} : \Db \rightarrow \Omega_n$ by $g_{n,r}(z) = g_n(rz)$. Then $g_{n,r}(0)=p_n$ and $g_{n,r}^\prime(0)=rv_n/\alpha_n$. So 
$$
k_{\Omega_n}(p_n;v_n) \leq \frac{\alpha_n}{r} =\frac{1}{r}k_{\Omega}(p_n;v_n)
$$
when $n \geq N_r$.

Since the Kobayashi distance on $\Omega$ is proper, $\Omega$ is a taut complex manifold. So $k_\Omega$ is continuous by~\cite[Proposition 2.3.34]{Aba89}. Hence 
\begin{align*}
\lim_{n \rightarrow \infty}k_{\Omega}(p_n;v_n) = k_\Omega(p;v)
\end{align*}
and so
\begin{align*}
\limsup_{n \rightarrow \infty} k_{\Omega_n}(p_n;v_n) \leq  \frac{1}{r} \limsup_{n \rightarrow \infty} k_{\Omega}(p_n,v_n) = \frac{1}{r}k_{\Omega}(p;v).
\end{align*}
Then since $r \in (0,1)$ is arbitrary,
\begin{equation}\label{right-ineq}
\limsup_{n \rightarrow \infty} k_{\Omega_n}(p_n;v_n) \leq k_{\Omega}(p;v).
\end{equation}
\end{proof}

\begin{lemma}\label{3-lem}
\begin{align*}
k_\Omega(p;v) \leq \liminf_{n \rightarrow \infty}k_{\Omega_n}(p_n;v_n).
\end{align*}
\end{lemma} 

\begin{proof}
Let $f_n:\mathbb D \rightarrow \Omega_n$ be a holomorphic map such that $f_n(0) = p_n$, $f'_n(0) = v_n /\alpha_n$, and
$$
\alpha_n = k_{\Omega_n}(p_n;v_n).
$$
Next pick $n_j \rightarrow \infty$ such that 
\begin{align*}
\liminf_{n \rightarrow \infty}k_{\Omega_n}(p_n;v_n) = \lim_{j \rightarrow \infty}k_{\Omega_{n_j}}(p_{n_j};v_{n_j}).
\end{align*}
By Lemma~\ref{2-lem}
\begin{equation}\label{eq:upper_bd_alpha}
\alpha : = \liminf_{n \rightarrow \infty}k_{\Omega_n}(p_n;v_n) \leq \limsup_{n \rightarrow \infty} k_{\Omega_n}(p_n;v_n)  < +\infty.
\end{equation}

\vspace{2mm}
\noindent{\it Claim.} Passing to a subsequence if necessary, we may assume that $f_{n_j}$ converges locally uniformly to a holomorphic map $f:\mathbb D \rightarrow \Omega$. 

\vspace{2mm}
\noindent{\it Proof of the Claim.} By Montel's theorem, it is enough to fix a compact set $Y \subset \Db$ and show that 
\begin{align*}
\sup_{j \geq 0} \max_{y \in Y} \norm{ f_{n_j}(y)} <+\infty.
\end{align*}
Suppose not, then after possibly passing to a subsequence  there exists $y_j \in Y$ such that 
\begin{align}\label{eq:norm_go_to_infty}
\lim_{j \rightarrow \infty} \norm{f_{n_j}(y_j)} =\infty.
\end{align}
 Let 
\begin{align*}
v_j := \frac{f_{n_j}(y_j)}{\norm{f_{n_j}(y_j)}}.
\end{align*}
By passing to a subsequence we can assume that $v_j \rightarrow v$. Next fix some 
\begin{align*}
q \in ( p + \Cb \cdot v) \cap \partial \Omega.
\end{align*}
Then there exists $q_j \in\partial \Omega_{n_j}$ such that $q_{j} \rightarrow q$. By passing to another subsequence we can assume that $p \in \Omega_{n_j}$ for all $j$. Since each $\Omega_{n_j}$ is convex, we can find an affine map $\ell_j : \Cb^d \rightarrow\Cb$ such that $\ell_j(q_j)=0$, $\ell_j(p)=i$, and ${ \rm Im}( \ell_j(\Omega_{n_j})) > 0$.  By passing to another subsequence we can suppose that $\ell_j$ converges locally uniformly to an affine map $\ell: \Cb^d \rightarrow \Cb$. Then $\ell(q)=0$, $\ell(p)=i$, and ${ \rm Im}(\ell(\Omega)) > 0$. 

Since $p_n \rightarrow p$, by passing to a subsequence we can suppose there exists some domain $D \subset \Omega$ such that $p_{n_j} \in D$ for all $j \geq 0$ and $\overline{D}$ is a compact subset of $\Omega$. Then by Lemma~\ref{lem:compact_inclusion}  and possibly passing to a subsequence, we can suppose that $\overline{D} \subset \Omega_{n_j}$ for all $j \geq 0$. 

Next let $\Hc := \{ z \in \Cb : { \rm Im}(z)>0\}$. Then 
\begin{align*}
K_{\Hc}( \ell_j( f_{n_j}(y_j)), i) &=K_{\Hc}( \ell_j( f_{n_j}(y_j)), \ell_j(p) )\leq K_{\Omega_{n_j}}(f_{n_j}(y_j), p) \\
& \leq K_{\Omega_{n_j}}(f_{n_j}(y_j), p_{n_j}) + K_{\Omega_{n_j}}(p_{n_j}, p) \\
& \leq  K_{\Db}(y_j, 0) + K_D(p_{n_j}, p).
\end{align*}
Since $Y \subset \Db$ is compact and $p_n \rightarrow p$, we see that 
\begin{align*}
\sup_{j \geq 0} K_{\Hc}( \ell_j( f_{n_j}(y_j)), i) \leq \sup_{y \in Y} K_{\Db}(y,0) + \sup_{j\geq0} K_D(p_{n_j}, p) <  +\infty
\end{align*}
which implies that 
\begin{align}\label{eq:norm_bd}
\sup_{j \geq 0} \abs{\ell_j( f_{n_j}(y_j))} < + \infty. 
\end{align}

Next pick $a,a_j \in \Cb$ and $b,b_j \in \Cb^d$ such that $\ell_j(z)= a_j + \ip{b_j,z}$ and $\ell(z) = a +\ip{b,z}$. Then $a_j \rightarrow a$ and $b_j \rightarrow b$. Further
\begin{align*}
\ell_j( f_{n_j}(y_j)) = a_j + \ip{b_j, f_{n_j}(y_j)} = a_j + \norm{f_{n_j}(y_j)} \ip{ b_j, v_j}.
\end{align*}
Then Equations~\eqref{eq:norm_go_to_infty} and~\eqref{eq:norm_bd} imply that 
\begin{align*}
\ip{b,v}= \lim_{j\rightarrow \infty} \ip{b_j,v_j} = 0. 
\end{align*}
Since $q=z_0+\lambda v$ for some $\lambda \in \Cb$ we have
\begin{align*}
i=\ell(z_0)=a+\ip{b,z_0} = a+\ip{b,q} + \ip{b,z_0-q} = \ell(q) - \lambda \ip{b,v}=0.
\end{align*}
So we have a contradiction. Thus by Montel's theorem and passing to a subsequence if necessary, we may assume that $f_{n_j}$ converges locally uniformly to a holomorphic map $f:\mathbb D \rightarrow \Omega$. This completes the proof of the claim.

Then 
\begin{align*}
f'(0) = \lim_{j \rightarrow \infty}f'_{n_j}(0) =  \lim_{j \rightarrow \infty} v_{n_j}/\alpha_{n_j}.
\end{align*}
By Equation~\eqref{eq:upper_bd_alpha}
\begin{align*}
\alpha = \lim_{j \rightarrow \infty} \alpha_{n_j} < +\infty
\end{align*}
and $v_n$ converges to $v$. So we must have $\alpha > 0$. Thus $f'(0) = v/\alpha$. Then
\begin{align*}
k_\Omega(p;v) \leq \alpha=\lim_{j \rightarrow \infty}\alpha_{n_j} = \liminf_{n \rightarrow \infty} k_{\Omega_n}(p_n;v_n).
\end{align*}
\end{proof}

Finally, it follows from Lemmas~\ref{2-lem} and \ref{3-lem} that
$$
\lim_{n \rightarrow \infty}k_{\Omega_n}(p_n;v_n) = k_{\Omega}(p;v). \qedhere
$$.

\end{document}